\documentclass[12pt]{article}
\usepackage[utf8]{inputenc}
\usepackage{amsmath}
\usepackage{amsthm}
\usepackage{amssymb}
\usepackage[affil-it]{authblk}

\newcommand{\PP}{\mathbb P}

\newtheoremstyle{theoremstyle}{}{}{\itshape}{}{\scshape}{.}{ }{\textbf{#1\ #2}}  
\theoremstyle{theoremstyle}
\newtheorem{theorem}{Theorem}[section]
\newtheorem{lem}{Lemma}[section]

\newtheorem{cor}{Corollary}[section]

\title{Friends of 12}
\author{\LARGE{Doyon Kim}\footnote{AMS(MOS) Subject Classification: 11A25}\footnote{This work was supported by NSF grant no. 1262930, and was completed during and after the 2015 summer Research Experience for Undergraduates in Algebra and Discrete Mathematics at Auburn University}
\\ \large{\textnormal{Auburn University}}
\\ \large\textnormal{dzk0028@auburn.edu}}
\date{July 2015}
\begin{document}

\maketitle
\abstract{A friend of \(12\) is a positive integer different from \(12\) with the same abundancy index. By enlarging the supply of methods of Ward [1], it is shown that
(i) if \(n\) is an odd friend of \(12\), then \(n=m^2\), where \(m\) has at least 5 distinct prime factors, including \(3\), and
(ii) if \(n\) is an even friend of \(12\) other than \(234\), then \(n=2\cdot q^e\cdot m^2\), in which \(q\) is a prime, \(e\) is a positive integer, \(29\leq q\equiv e\equiv 1\mod 4\), and \(m\) has at least 3 distinct odd prime factors, one of which is \(3\), and the other, none equal to \(q\), are greater or equal to \(29\).}

\section{The Abundancy Index}
Let \(\PP\) denote the set of positive integers. For \(n\in\PP\), let \(\sigma(n)\) denote the sum of all positive divisors of \(n\), including \(n\) itself. It is well know, and not hard to see, that \(\sigma\) is \textit{weakly multiplicative}: That is, if \(m,n\in\PP\) and \(\gcd(m,n)=1\), then \(\sigma(mn)=\sigma(m)\sigma(n)\). Therefore, if \(q_1,\dots,q_t\in\PP\) are distinct primes, and \(e_1,\dots,e_t\in\PP\), then \(\sigma(\prod_{i=1}^{t} q_i^{e_i})=\prod_{i=1}^{t}\sigma(q_i^{e_i})=\prod_{i=1}^{t}(\sum_{j=0}^{e_i} q_i^j)=\prod_{i=1}^{t}\frac{q_i^{e_i+1}-1}{q_i-1}\). For instance, \(\sigma(12)=\sigma(3)\sigma(4)=(1+3)(1+2+4)=28\). \par
The abundancy ratio, or abundancy index, of \(n\in\PP\), is \(I(n)=\frac{\sigma(n)}{n}\). From previous remarks about \(\sigma\) we have the following facts about properties of the abundancy index. These are also mentioned in [1].
\begin{itemize}
\item[1.] \(I\) is weakly multiplicative.
\item[2.] If \(q_1,\dots,q_t\in\PP\) are distinct primes, and \(e_1,\dots,e_t\in\PP\), then \(I(\prod_{i=1}^{t} q_i^{e_i})=\prod_{i=1}^{t}I(q_i^{e_i})=\prod_{i=1}^{t}\frac{\sum_{j=0}^{e_i} q_i^j}{q_i^{e_i}}=\prod_{i=1}^{t}\frac{q_i^{e_i+1}-1}{q_i^{e_i}(q_i-1)}\).
\item[3.] If \(q\in\PP\) is a prime, then, as \(e\in\PP\) increases from \(1\), \(I(q^e)\) is strictly increasing, from \(\frac{q+1}{q}\), and tends to \(\frac{q}{q-1}\) as \(e\to\infty\).
\item[4.] If \(e\in\PP\), as \(q\) increases among the positive primes, \(I(q^e)\) is strictly decreasing. 
\item[5.] If \(m,n\in\PP\) and \(m\mid n\), then \(I(m)\leq I(n)\), with equality only if \(m=n\).
\end{itemize} \par
Interest in the abundancy index arises from interest in perfect numbers, positive integers with abundancy index \(2\). Mathematicians have been curious about perfect numbers since antiquity, and the abundancy index offers a context within which to study them indirectly. Perhaps asking questions about the abundancy index will lead to the development of theory applicable to question about the perfect numbers. See [1] for more references on the abundancy index.
\section{Friends}
\ \ \ \ Positive integers \(m\) and \(n\) are friends if and only if \(m \neq n\) and \(I(m)=I(n)\). Thus, different perfect numbers are friends. As in [1], it is easy to see that \(1\) has no friend, and that no prime power has a friend. It is not known if any positive integer has infinitely many friends. \par
Every element of \(\{1,\dots,9\}\) is a prime power except for \(1\) and \(6\). \(I(1)=1\), and \(1\) has no friend; \(6\) is the smallest perfect number. (Actually, everything we would want to know about even perfect numbers, except whether or not there are infinitely many of them, is known, thanks to Euler.) Therefore, the first frontier of the study of friends is made up of the integers \(10\), \(12\), \(14\), \(15\), and \(18\). \par
In [1], Ward took on \(10\), and proved, among other things, that any friend of \(10\) must be a square with at least \(6\) distinct prime factors, including \(5\), the smallest. It is still unknown whether or not \(10\) has a friend; however, we feel that Ward has done service by pioneering methods other than "computer search" for hunting for friends of given integers. \par
The aim here is to apply Ward's methods, with a few new tricks thrown in, to the search for friends of \(12\). Computer search has already discovered \(234\) to be a friend of \(12\). This discovery can also be made rationally, by Ward-like arguments. Using these arguments, with a few twists, we shall obtain a theorem about the friends of \(12\) similar to Ward's theorem about friends of \(10\). The following lemma will be useful. The proof is straightforward.
\begin{lem}
If \(p\) is an odd prime and \(e\in\PP\), then \(\sigma(p^e)\) is odd if and only if \(e\) is even. If \(p\equiv 3\mod 4\) and \(e\) is odd, then \(4\mid \sigma(p^e)\). If \(p\equiv 1\mod 4\) then \(\sigma(p^e)\equiv e+1\mod 4\).
\end{lem}
\begin{cor}
If \(n\in\PP\) and both \(n\) and \(\sigma(n)\) are odd, then \(n=k^2\) for some \(k\in\PP\). If \(n\) is even and \(\sigma(n)\) is odd, then \(n=2^fk^2\) for some \(f,k\in\PP\).
\end{cor}

\section{Friends of 12}
\begin{theorem} If \(n\notin\{12,234\}\) is a positive integer such that \(I(n)=I(12)=I(234)=\frac{7}{3}\) and \(n\) is odd, then \(n=3^{2a}\prod_{i=1}^{k} p_i^{2e_i}\), where \(p_1, \dots, p_k\) are distinct primes greater than \(3\), \(a, e_1, \dots, e_k \in \PP\), and \(k\geq 4\). 
If \(n\notin\{12,234\}\), \(I(n)=\frac{7}{3}\) and \(n\) is even, then \(n=2\cdot 3^{2a}\cdot q^e\prod_{i=1}^{k} p_i^{2e_i}\), where \(q,p_1, \dots, p_k\) are distinct primes greater than or equal to 29, \(a,e, e_1, \dots, e_k \in \PP\), \(a\geq 3\), \(k\geq 2\), and \(q\equiv e\equiv 1 \mod 4\).
\end{theorem}
\begin{proof} 
Preassuming its existence, let \(n\notin\{12,234\}\) be a positive integer such that \(I(n)=\frac{7}{3}\). Since \(\frac{\sigma(n)}{n}=I(n)=\frac{7}{3}\), \(3\sigma(n)=7n\) and therefore \(3\mid n\). \par

Suppose \(n\) is odd. Since \(3\sigma(n)=7n\) and \(7n\) is odd, \(\sigma(n)\) is odd. Since both \(n\) and \(\sigma(n)\) are odd, \(n\) must be a square, by Corollary 2.1. Therefore, \(n=3^{2a}m^2\), where \(a,m\in \PP\), \(m\) is odd and \(3\nmid m\). \par 

If \(m\) is 1, then \(I(n)=I(3^{2a})<\frac{3}{2}<\frac{7}{3}\)=I(n), a contradiction. Therefore \(m>1\), so \(n=3^{2a}\prod_{i=1}^{k} p_i^{2e_i}\), where \(p_1, \dots, p_k\) are distinct primes greater than \(3\), \(a, e_1, \dots, e_k \in \PP\), and \(k\geq 1\).\par

 If \(k\leq 2\), then \(I(n)\leq I(3^{2a}5^{2e_1}7^{2e_2})<\frac{3}{2}\frac{5}{4}\frac{7}{6}<\frac{7}{3}\). Therefore \(k\geq 3\). \par

If \(k=3\), then since \[I(3^{2a}5^{2e_1}7^{2e_2}17^{2e_3})<\frac{3}{2}\frac{5}{4}\frac{7}{6}\frac{17}{16}<\frac{7}{3}\] and \[I(3^{2a}5^{2e_1}11^{2e_2}13^{2e_3})<\frac{3}{2}\frac{5}{4}\frac{11}{10}\frac{13}{12}<\frac{7}{3},\] \(p_1=5\), \(p_2=7\), and \(p_3\in \{11,13\}\). \par

Verify that \[I(3^45^27^211^2)=\frac{121}{81}\frac{31}{25}\frac{57}{49}\frac{133}{121}>\frac{7}{3}.\] Thus if \(n=3^{2a}5^{2e_1}7^{2e_2}11^{2e_3},\) then \(a=1\). But then \(7n=3\sigma(n)=3\sigma(3^25^{2e_1}7^{2e_2}11^{2e_3})=3\cdot 13\cdot\sigma(5^{2e_1}7^{2e_2}11^{2e_3})\) would imply that \(13\mid n\), which does not hold. Therefore \(p_3\neq 11\). \par

Similarly, verify that \(I(3^65^27^213^2)=\frac{1093}{729}\frac{31}{25}\frac{57}{49}\frac{183}{169}>\frac{7}{3}\). Thus if  \(n=3^{2a}5^{2e_1}7^{2e_2}13^{2e_3}\), \(a\in\{1,2\}\). If \(a=1\) then \[I(n)=I(3^25^{2e_1}7^{2e_2}13^{2e_3})=\frac{13}{9}I(5^{2e_1}7^{2e_2}13^{2e_3})<\frac{13}{9}\frac{5}{4}\frac{7}{6}\frac{13}{12}<\frac{7}{3}\] so \(a\neq 1\). If \(a=2\) then \(7n=3\sigma(n)=3\sigma(3^45^{2e_1}7^{2e_2}13^{2e_3})=3\cdot 121\cdot \sigma(5^{2e_1}7^{2e_2}13^{2e_3})\) would imply that \(11\mid n\), which does not hold. Therefore \(a\notin \{1,2\}\) and we conclude that \(k\geq 4\). \par

Therefore, if \(n\) is odd and \(I(n)=\frac {7}{3}\), then \(n=3^{2a}\prod_{i=1}^{k} p_i^{2e_i},\) where \(p_1, \dots, p_k\) are distinct primes greater than \(3\), \(a, e_1, \dots, e_k \in \PP\), and \(k\geq 4\). \par

Now, suppose \(n\) is even. Since \(3\mid n\), \(2^2\) does not divide \(n\) because if it did then \(12\) would divide \(n\) and we would have \(I(n)>I(12)=\frac{7}{3}\). So \(n=2\cdot 3^b\cdot m\), where \(b, m\in \PP\), \(m\) is odd and \(3\nmid m\). \par

If \(m=1\), then \(I(n)=I(2\cdot 3^b)=I(2)I(3^b)<\frac{3}{2}\frac{3}{2}<\frac{7}{3}\). Therefore \(m>1\), and \(n=2\cdot 3^b\prod_{i=1}^{k} q_i^{f_i}\) where \(b, f_1, \dots, f_k \in \PP\), \(k\geq 1\) and \(q_1<\dots<q_k\) are distinct primes greater than \(3\). \par

Since \(7n=3\sigma(n)=3\sigma(2\cdot 3^b\prod_{i=1}^{k} q_i^{f_i})=3\cdot 3\cdot\sigma(3^b\prod_{i=1}^{k} q_i^{f_i})\), \(2\mid \sigma(3^b\prod_{i=1}^{k} q_i^{f_i})\) but \(4=2^2\nmid \sigma(3^b\prod_{i=1}^{k} q_i^{f_i})\). \par

If all of \(b, f_1, \dots, f_k\) are even then \(\sigma(3^b\prod_{i=1}^{k} q_i^{f_i})\) would be odd, and if two or more of \(b, f_1, \dots, f_k\) are odd then we would have \(4\mid\sigma(3^b\prod_{i=1}^{k} q_i^{f_i})\). Therefore exactly one of \(b, f_1, \dots, f_k\) is odd. Further, by Lemma 2.1, if \(q\equiv 3\mod 4\) and is a prime and \(e\equiv 1\mod 2\) then \(\sigma(q^e)\equiv 0 \mod 4\) and if \(q\equiv1\mod 4\) and \(e\equiv 3\mod 4\) then \(\sigma(q^e)\equiv 0 \mod 4\). Therefore \(b\) must be even, and exactly one of \(f_i, 1\leq i \leq k\) is congruent to \(1 \mod 4\). Also, for such \(f_i\) its corresponding prime divisor \(q_i\)  is also congruent to \(1\mod 4\). That is, \(n=q^e(3m)^2\), where \(m\) is a positive odd integer, \(q\) is prime that does not divide \(m\), and \(q\equiv e\equiv 1\mod 4\). \par

If \(b=2\), then \(7n=3\sigma(n)=3\sigma(2\cdot 3^2\prod_{i=1}^{k} q_i^{f_i})=3\cdot 3\cdot 13\cdot\sigma(\prod_{i=1}^{k} q_i^{f_i})\) would imply that \(13\mid n\). But then \(234=2\cdot 3^2\cdot 13\) would divide \(n\), and therefore \(I(n)>I(234)=\frac{7}{3}\). Therefore, \(b\neq 2\). Also, if \(b=4\), then \(7n=3\sigma(n)=3\sigma(2\cdot 3^4\prod_{i=1}^{k} q_i^{f_i})=3\cdot 3\cdot 121\cdot\sigma(\prod_{i=1}^{k} q_i^{f_i})\) would imply that \(11\mid n\). But then it follows that \(I(n)\geq I(2\cdot 3^4\cdot 11)=\frac{3}{2}\frac{121}{81}\frac{12}{11}>\frac{7}{3}\), so we conclude that \(b\neq 4\). Therefore \(b\geq 6\). \par

Since \(\frac{7}{3}=I(n)\geq I(2\cdot 3^6\cdot q_1)=\frac{3}{2}\frac{3^7-1}{2\cdot3^6}\frac{q_1+1}{q_1}\), it follows that \(\frac{q_1+1}{q_1}\leq \frac{7}{3}\frac{2}{3}\frac{2\cdot3^6}{3^7-1}\) and thus \(q_1>26\). Therefore \(q_1\geq 29\). \par

If \(k=1\), then \(I(n)=I(2\cdot 3^b\cdot q_1^{f_1})\leq I(2\cdot 3^b\cdot 29^{f_1})=I(2)I(3^b\cdot 29^{f_1})<\frac{3}{2}\frac{3}{2}\frac{29}{28}<\frac{7}{3}\). Therefore \(k\geq 2\). \par

Suppose \(k=2\). Since \[\frac{7}{3}=I(n)=I(2\cdot 3^b\cdot q_1^{f_1}\cdot q_2^{f_2})> I(2\cdot 3^6\cdot q_1\cdot q_2)=\frac{3}{2}\frac{3^7-1}{2\cdot 3^6}\frac{q_1+1}{q_1}\frac{q_2+1}{q_2},\]
 we have \[\frac{q_1+1}{q_1}\frac{q_2+1}{q_2}\leq\frac{7}{3}\frac{2}{3}\frac{2\cdot3^6}{3^7-1}=\frac{1134}{1093}.\] Also, since \[\frac{7}{3}=I(n)=I(2\cdot 3^b\cdot q_1^{f_1}\cdot q_2^{f_2})=I(2)I(3^b\cdot q_1^{f_1}\cdot q_2^{f_2})<\frac{3}{2}\frac{3}{2}\frac{q_1}{q_1-1}\frac{q_2}{q_2-1},\] we have \[\frac{q_1}{q_1-1}\frac{q_2}{q_2-1}>\frac{7}{3}\frac{2}{3}\frac{2}{3}=\frac{28}{27}.\] \par

\begin{center}
\begin{align*}
\frac{q_1+1}{q_1}\frac{q_2+1}{q_2}\leq\frac{1134}{1093} &\iff (q_1+1)(q_2+1)\leq\frac{1134}{1093} q_1q_2 \\
&\iff \frac{41}{1093}q_1q_2-q_1-q_2-1\geq 0 \\
&\iff (\frac{41}{1093} q_1-1)(q_2-\frac{1093}{41})-\frac{1134}{41}\geq 0 \\
&\iff q_2 \geq \frac{1239462}{1681q_1-44813}+\frac{1093}{41},
\end{align*}
\end{center}
and
\begin{center}
\begin{align*}
\frac{q_1}{q_1-1}\frac{q_2}{q_2-1}>\frac{28}{27} &\iff \frac{27}{28} q_1q_2>(q_1-1)(q_2-1) \\
&\iff \frac{1}{28} q_1q_2-q_1-q_2+1<0 \\
&\iff (\frac{1}{28} q_1-1)(q_2-28)-27<0 \\
&\iff q_2<\frac{27}{\frac{1}{28} q_1-1}+28=\frac{756}{q_1-28}+28,
\end{align*}
\end{center}
so we have the inequality
\[
\frac{1239462}{1681q_1-44813}+\frac{1093}{41}\leq q_2< \frac{756}{q_1-28}+28. \tag{1}
\] \par

Since \(29\leq q_1<q_2\) and \(q_2<\frac{756}{q_1-28}+28\), \(29\leq q_1<\frac{756}{q_1-28}+27\). To satisfy this inequality, \(q_1\) must be in the interval \([29,55]\) and therefore \(q_1\in \{29,31,37,41,43,47,53\}\). \par

If \(b=6\), then \(7n=3\sigma(n)=3\sigma(2\cdot 3^6\cdot q_1^{f_1}\cdot q_2^{f_2})=3\cdot 3\cdot 1093\cdot\sigma(q_1^{f_1}\cdot q_2^{f_2})\) so 1093 would divide \(n\). But since \(q_2< \frac{756}{q_1-28}+28\) and \(q_1\in \{29,31,37,41,43,47,53\}\), \(q_2<784\) and so \(1093\notin \{q_1,q_2\}\). Therefore \(b\neq 6\). If \(b=8\), then since \(13\mid 9841=\sigma(3^8)\) and \(\sigma(3^8)\mid 3\sigma(n)=7n\), \(13\) would divide \(n\). However, since \(13<29\leq q_1<q_2\), \(13\) does not divide \(n\) and therefore \(b\neq 8\). Likewise,  if \(b=10\), then since \(23\mid \sigma(3^{10})=88573\), \(23\) would divide \(n\). Since \(23<29\leq q_1<q_2\), \(23\nmid n\) and therefore \(b\neq 10\). Therefore \(b\geq 12\). \par

Suppose \(q_1=29\). By the inequality (1), \(341<q_2<784\). If \(f_1=1\), then \(7n=3\sigma(n)=3\sigma(2\cdot 3^b\cdot 29^1\cdot q_2^{f_2})=3\cdot 3\cdot \sigma(3^b)\cdot 30\cdot \sigma(q_2^{f_2})\) so \(5\) would divide \(n\). Therefore \(f_1\neq 1\). Likewise, \(f_1\neq 2\) because \(13\mid 871=\sigma(29^2)\) but \(13\nmid n\). Therefore, \(f_1\geq 3\). So we have \(\frac{7}{3}=I(n)\geq I((2\cdot 3^{12}\cdot 29^3\cdot q_2)=\frac{3}{2}\frac{3^{13}-1}{2\cdot 3^{12}}\frac{29^4-1}{28\cdot 29^3}\frac{q_2+1}{q_2}\), and therefore \(781<q_2\). Since there is no prime between \(781\) and \(784\), this is a contradiction and we conclude that \(q_1\neq 29\).\par

Suppose \(q_1=31\). By (1), \(196<q_2<280\); \(f_1\) must be even because \(q_1\equiv 3\mod 4\), and \(f_1\neq 2\) because if \(f_1=2\) then \(7n=3\sigma(n)=3\sigma(2\cdot 3^b\cdot 31^2\cdot q_2^{f_2})=3\cdot 3\cdot \sigma(3^b)\cdot 993\cdot \sigma(q_2^{f_2})\) would imply that \(331\mid n\), which does not hold because \(31=q_1<q_2<280<331\). So \(f_1\geq 4\), and we have \(\frac{7}{3}=I(n)\geq I(2\cdot 3^{12}\cdot 31^4\cdot q_2)=\frac{3}{2}\frac{3^{13}-1}{2\cdot 3^{12}}\frac{31^5-1}{30\cdot 31^4}\frac{q_2+1}{q_2}\), therefore \(q_2>278\). This is a contradiction because there is no prime between \(278\) and \(280\). Therefore \(q_1\neq 31\). \par

Suppose \(q_1=37\). By (1), \(97<q_2<112\). \(f_1\neq 1\) because \(19\mid \sigma(37^1)\) but \(19\nmid n\). So \(f_1\geq 2\), and therefore \(\frac{7}{3}=I(n)\geq I(2\cdot 3^{12}\cdot 37^2\cdot q_2)=\frac{3}{2}\frac{3^{13}-1}{2\cdot 3^{12}}\frac{37^3-1}{36\cdot 37^2}\frac{q_2+1}{q_2}\). This inequality is valid only if \(q_2>110\), and this is a contradiction because there is no prime between \(110\) and \(112\). Therefore \(q_1\neq 37\). \par

Suppose \(q_1=41\). By (1), \(78<q_2<86\), or \(q_2\in \{79,83\}\). Since \(79\equiv 83\equiv 3\mod 4\) and \(q_1=41\equiv 1\mod 4\), \(f_1\equiv 1\mod 4\) and \(f_2\) is even. Verify that \(I(2\cdot 3^{12}\cdot 41\cdot 79^2)>\frac{7}{3}\), therefore \(q_2\neq 79\). Verify also that \(I(2\cdot 3^{12}\cdot 41^5\cdot 83^2)>\frac{7}{3}\), so if \(n=2\cdot 3^b\cdot 41^{f_1}\cdot 83^{f_2}\) then \(f_1=1\). But then \(I(n)=I(2\cdot 3^b\cdot 41\cdot 83^{f_2})=I(2)I(41)I(3^b83^{f_2})<\frac{3}{2}\frac{42}{41}\frac{3}{2}\frac{83}{82}<\frac{7}{3}\), therefore \(q_2\neq 83\). So \(q_2\notin \{79,83\}\), and this is a contradiction so we conclude that \(q_1\neq 41\). \par

If \(q_1=43\), \(71<q_2<78\) by (1) so \(q_2=73\), and \(f_1\) is even because \(q_1\equiv 3\mod 4\). It would follow that \(I(n)\geq I(2\cdot 3^{12}\cdot 43^2\cdot 73)=\frac{3}{2}\frac{3^{13}-1}{2\cdot 3^{12}}\frac{43^3-1}{42\cdot 43^2}\frac{74}{73}>\frac{7}{3}\); therefore \(q_1\neq 43\). \par

If \(q_1=47\), then  \(62<q_2<68\) by (1) so \(q_2=67\). But since \(47\equiv 67\equiv 3\mod 4\), \(q_1\neq 47\), for if \(q_1=47\) then \(n=2\cdot 3^b\cdot 47^{f_1}\cdot 67^{f_2}\) and \(\sigma(n)=\sigma(2\cdot 3^b\cdot 47^{f_1}\cdot 67^{f_2})\) would either be odd or divisible by 4. \par

If \(q_1=53\), then \(54<q_2<59\). Since there is no prime between \(54\) and \(59\), \(q_1\neq 53\). \par

Therefore, \(q_1\notin \{29,31,37,41,43,47,53\}\). This is contradictory to the assumption that \(k=2\), so we conclude that \(k\geq 3\). \par

Therefore, \(n=2\cdot 3^{2a}\cdot q^e\prod_{i=1}^{k} p_i^{2e_i}\), where \(q,p_1, \dots, p_k\) are distinct primes greater or equal to 29, \(a,e, e_1, \dots, e_k \in \PP\), \(a\geq 3\), \(k\geq 2\), and \(q\equiv e\equiv 1 \mod 4\).

\end{proof}

\end{document}